\documentclass[11pt]{amsart}
\usepackage{amsthm}
\usepackage{amsmath}
\usepackage{amssymb}
\usepackage{enumerate}

\newtheorem{theorem}{Theorem}[section]
\newtheorem{lemma}[theorem]{Lemma}

\newtheorem{corollary}[theorem]{Corollary}
\newtheorem{definition}[theorem]{Definition}

\newtheorem{thm}{Theorem}

\newcommand \Z {\mathbb Z}

\newcommand \opi {{\bf{O}}_{\pi}}

\newcommand \opd {{\bf{O}}_{p'}}
\newcommand \cent {{\bf{C}}}
\newcommand \F {\mathbb F}

\newcommand \irr {\textup{Irr}}
\newcommand \irrg {\textup{Irr}(G)}
\newcommand \ibr {{\textup{IBr}}_p}
\newcommand \ibrg {{\textup{IBr}}_p(G)}

\newcommand \ngd {{\bf{N}}_G(D)}

\newcommand \norm {{\bf{N}}}

\newcommand \dvd {\hbox {\big|}}
\newcommand \ndvd {\hbox {/}\kern-5pt\dvd}

\newcommand{\IBr}[1]{{\rm IBr}_p (#1)}
\def\irr#1{{\rm  Irr}(#1)}
\def\irri#1#2{{\rm Irr}_{(#1)} (#2)}
\def \< {\langle}
\def \> {\rangle}

\begin{document}

\title[Counting characters]{Counting characters in blocks of solvable groups with abelian defect group}


\author{James P. Cossey}
\author{Mark L. Lewis}

\address{Department of Theoretical and Applied Mathematics,
University of Akron,
Akron, OH 44325}

\email{cossey@uakron.edu}

\address{Department of Mathematical Sciences, Kent State
University, Kent, OH 44242}

\email{lewis@math.kent.edu}


\begin{abstract}  If $G$ is a solvable group and $p$ is a prime, then the Fong-Swan theorem shows that given any irreducible Brauer character $\varphi$ of $G$, there exists a character $\chi \in \irrg$ such that $\chi^o = \varphi$, where $^o$ denotes the restriction of $\chi$ to the $p$-regular elements of $G$.  We say that $\chi$ is a {\it{lift}} of $\varphi$ in this case.  It is known that if $\varphi$ is in a block with abelian defect group $D$, then the number of lifts of $\varphi$ is bounded above by $|D|$.  In this paper we give a necessary and sufficient condition for this bound to be achieved, in terms of local information in a subgroup $V$ determined by the block $B$.  We also apply these methods to examine the situation when equality occurs in the $k(B)$ conjecture for blocks of solvable groups with abelian defect group.

\end{abstract}

\maketitle

\section{Introduction}

Let $G$ be a finite group and $p$ a fixed prime.  One of the main areas of research in the representation theory of finite groups is the interplay between the representation theory of $G$ over the complex numbers and the representation theory of $G$ over an algebraically closed field $\F$ of characteristic $p$.  In particular, we will look at the connection between the set $\irrg$ of ordinary irreducible characters of $G$, and the set $\ibrg$ of irreducible Brauer characters of $G$, which are complex-valued class functions on the set of $p$-regular elements of $G$ that in some sense describe the irreducible representations of $G$ over $\F$.  Let $\alpha^o$ denote the restriction of the complex valued class function $\alpha$ to the $p$-regular elements of $G$.  It has long been known that in this case, if $\chi \in \irrg$, then we have $$\chi^o = \sum_{\varphi \in \ibrg} d_{\chi \varphi} \varphi,$$ where the $\{d_{\chi \varphi} \}$ are uniquely defined non-negative integers, called the decomposition numbers of $\chi$.  (See \cite{blocks} for a more thorough treatment of the basic facts about Brauer characters and decomposition numbers of finite groups.)

If $G$ is solvable, there is more that can be said.  The Fong-Swan theorem (see \cite{bpi} or \cite{blocks} for a proof of the Fong-Swan theorem) says that if $\varphi \in \ibrg$, then there is an ordinary irreducible character $\chi$ of $G$ such that $\chi^o = \varphi$.  In this case, we say that $\chi$ is a lift of $\varphi$, and we define $L_{\varphi} = \{\chi \in \irrg \mid \chi^o = \varphi \}$ to be the set of lifts of $\varphi$.  Recently the authors (along with others) have been examining the structure of $L_{\varphi}$ and determining upper bounds on the size of this set (see \cite{bounds} and \cite{CoLeNa}).

Now, assume $G$ is solvable and $B$ is a block of $G$ with abelian defect group $D$ (again, see \cite{blocks} for the definition and properties of the defect group of a block).  It is known in this case (either as an immediate consequence of the $k(B)$ theorem, or, perhaps more easily, as a consequence of Corollary B of \cite{CoLeNa}, since here the vertex subgroup is equal to the defect group) that if $\varphi \in \ibr(B)$, then $$|L_{\varphi}| \leq |D|.$$  In this paper we examine a necessary and sufficient condition for equality to occur in this inequality.  In particular, given any block $B$ of a solvable group $G$ with defect group $D$, we will define a certain canonical subgroup $V_B$ containing (a conjugate of) $D$ as a Sylow $p$-subgroup, called the Fong subgroup of $B$.  Our first main theorem is the following:

\begin{thm}\label{necsuf}  Let $G$ be a solvable group and $B$ a block of $G$ with abelian defect group, let $V$ be the Fong subgroup for $B$, and let $D$ be a defect group for $B$ contained in $V$.  Suppose $\varphi \in \ibr(B)$.  Then $|L_{\varphi}| = |D|$ if and only if $D \leq \Z(\norm_V(D))$.
\end{thm}

One might perhaps wonder if it would be more natural to consider the subgroup $\ngd$ rather than $\norm_V(D)$ in the statement of Theorem \ref{necsuf}.  We will show that there is in fact a counterexample to show that the statement does not hold if we replace $\norm_V(D)$ with $\ngd$.

The assumption that $|L_{\varphi}| = |D|$ in the above theorem may seem a bit strong, in light of the $k(B)$ problem.  Our next result gives a number of equivalent conditions, in the case that the defect group is abelian.  Given a Brauer character $\varphi$ of $G$, we let $\irr{G : \varphi} = \{ \chi \in \irrg \mid d_{\chi \varphi} \neq 0 \}$.

\begin{thm}\label{3conditions}  Let $G$ be a solvable group and $B$ a block of $G$ with abelian defect group $D$ and let $\varphi \in \ibrg$.  Then the following conditions are equivalent:

\begin{enumerate}
\item $|L_{\varphi}| = |D|$.

\item $\irr{G : \varphi} = L_{\varphi}$.

\item $|\irr{G : \varphi}| = |D|$.
\end{enumerate}
\end{thm}

Theorem \ref{necsuf} is example of a result showing that a ``global'' condition is equivalent to a ``local'' condition.  This has been a common theme in looking at lifts.  In Theorem \ref{3conditions}, we show that the global condition of Theorem \ref{necsuf} is equivalent to two other global conditions, one of which does not involve any discussion of lifts.  To our knowledge, this is the first time a global condition on lifts has been shown to be equivalent to a global condition that does not involve lifts.

Note that the Fong-Swan theorem actually applies to $p$-solvable groups.  Furthermore, most of our results also apply not just to solvable groups, but to $p$-solvable groups, and will be proved in that setting.  However, one direction of Theorem \ref{necsuf} requires the group $G$ (or, more specifically, a $p$-complement in $G$) to be solvable.

 We feel it is important to point out that the proofs of the main results are ``character theoretic'' in nature, in that the proofs only use the characters and not the underlying representations.  In particular this means that all of the results in this paper also hold in the setting of Isaacs' ``$\pi$-partial'' characters of $\pi$-separable groups (see \cite{bpi}) and the proofs would be the same (in particular, see \cite{slattery1} and \cite{slattery2} for a discussion of the $\pi$-block theory).  Again, we still would need the solvable hypothesis for the one direction of Theorem \ref{necsuf} even in this setting.

We will conclude with a section that briefly discusses how our results fit into the context of nilpotent blocks.  Some of these results are known for nilpotent $p$-blocks, though there does not yet seem to be a theory of nilpotent $\pi$-blocks.  In the final section we will show how our results yield a version of the Broue-Puig theorem for nilpotent $\pi$-blocks with abelian defect group, while simultaneously giving a new proof of that Broue-Puig theorem for nilpotent $p$-blocks of solvable groups with abelian defect group.

\section{The pair $(V_B, B_V)$ and character triple isomorphisms}

Let $B$ be a block of a solvable group $G$.  In this section we will define the Fong pair $(V_B, B_V)$ and discuss some of the basic properties of this pair.  The following definition is essentially simply naming the end result of the Fong reduction in a solvable group (see Theorems 9.14 and 10.20 of \cite{blocks} for more details).

\begin{definition}\label{fongsubgroup}  Let $G$ be a solvable group and let $B$ be a block of $G$.  Write $M = \opd(G)$, and let $\alpha \in \irr{M}$ be covered by $B$.  If $\alpha$ is invariant in $G$, then we define the Fong pair $(V_B, B_V)$ to be $(G, B)$.  If $\alpha$ is not invariant in $G$, then we let $T$ be the stabilizer of $\alpha$ in $G$, and note that by the Fong-Reynolds theorem (see Theorem 9.14 of \cite{blocks}), there exists a block $B_1$ of $T$ covering $\alpha$ such that induction is a bijection from $B_1$ to $G$.  In this case we define the Fong pair $(V_B, B_V)$ for $B$ to be a Fong pair for $B_1$.
\end{definition}

It is perhaps worth mentioning in the above definition that in general one could have $\opd(T) > \opd(G)$.  We say $V_B$ is the Fong subgroup for $B$, and we say $B_V$ is the corresponding block.  We note some immediate consequences of the above definition.  By the Fong-Reynolds theorem, any Sylow $p$-subgroup of $V_B$ is a defect group for the block $B$.  Moreover, obviously one could have chosen a different character $\alpha$ of $M$ covered by $B$, but a different choice would have yielded a conjugate Fong pair.  Finally, note that induction preserves decomposition numbers.   In other words, with the above notation, if $\chi \in \irr{B}$ and $\varphi \in \IBr{B}$ are induced by $\chi_1 \in \irr{B_1}$ and $\varphi_1 \in \IBr{B_1}$, respectively, then $d_{\chi \varphi} = d_{\chi_1 \varphi_1}$.

The following lemma is clear from the above discussion.

\begin{lemma}\label{fongred}  Let $G$ be a solvable group and $B$ a block of $G$ with Fong pair $(V_B, B_V)$.  Then induction is a bijection from $B_V$ to $B$, and the induction map preserves the decomposition numbers $d_{\chi \varphi}$.
\end{lemma}

This leads to the following immediate corollary.

\begin{corollary}\label{blockcomparison} Let $G$ be a solvable group and $B$ a block of $G$ with Fong pair $(V_B, B_V)$.  Let $\varphi \in \IBr{B}$ and let $\widehat{\varphi}$ denote the corresponding character of $V_B$.  Then induction is a bijection from $L_{\widehat{\varphi}}$ to $L_{\varphi}$ and from $\irr{V_B : \widehat{\varphi}}$ to $\irr{G : \varphi}$.
\end{corollary}

We will need to use the theory of character triples (see Chapter 11 of \cite{text} and Chapter 8 of \cite{blocks} for more details).  The following lemma is essentially contained in the proof of Corollary B of \cite{CoLeNa}.  Notice that for isomorphic character triples whose normal subgroups are $p'$-groups that this provides a correspondence of the Brauer characters in the blocks covering the character of the normal subgroup.

\begin{lemma} \label{triples}
Let $p$ be a prime, let $(G,N,\theta)$ and $(\Gamma,M,\gamma)$ be isomorphic character triples where $N$ and $M$ are $p'$-groups and $G$ or $\Gamma$ is a $p$-solvable group.  Suppose $\chi \in \irr{G \mid \theta}$ and $\psi \in \irr{\Gamma \mid \gamma}$ correspond under the character triple isomorphism.
\begin{enumerate}
\item Then $\chi^o$ is irreducible if and only if $\psi^o$ is irreducible.
\item The character triple isomorphism yields a bijection from $\ibr(G \mid \theta)$ to $\ibr(\Gamma \mid \gamma)$.
\item If $\chi^o$ is irreducible, then the character triple isomorphism restricts to bijections from $\irr{G : \chi^o}$ to $\irr{\Gamma : \varphi^o}$ and from $L_{\chi^o}$ to $L_{\psi^o}$.
\end{enumerate}
\end{lemma}

\begin{proof}
Notice that $G$ is $p$-solvable if and only if $\Gamma$ is $p$-solvable since $N$ and $M$ are both $p'$-groups.  Thus, assuming one is $p$-solvable is sufficient to see that they both are.  Write $*$ to represent the bijection of characters induced by the character triple isomorphism, so that $\psi = \chi^*$.  Take $H$ to be a Hall $p$-complement of $G$, and let $H^*$ correspond to $H$, and note that $H^*$ will be a Hall $p$-complement of $\Gamma$.

By the Fong-Swan theorem, $\chi^o$ is reducible if and only if there exist characters $\alpha, \beta \in \irr{G}$ such that $\chi^o = \alpha^o + \beta^o$. This occurs if and only if $\chi_H = \alpha_H + \beta_H$. Using the character triple isomorphism, this is equivalent to $\psi_{H^*} = (\alpha^*)_{H^*} + (\beta^*)_{H^*}$ and to $\psi^o = (\alpha^*)^o + (\beta^*)^o$.  We conclude that $\chi^o$ is irreducible if and only if $\psi^o$ is irreducible, and we have proved the first statement.

Suppose that $\chi^o$ is irreducible.  Notice that $\eta \in \irr G$ is a lift of $\chi^o$ if and only if $\eta_H = \chi_H$.   Similarly, $\eta^*$ is a lift of $\psi^o$ if and only if $(\eta^*)_H = \psi_H$.  It follows that $\eta$ is a lift of $\chi^o$ if and only if $\eta^*$ is a lift of $\psi^o$, and thus the character triple isomorphism is a bijection from $L(\chi^o)$ to $L(\psi^o)$, and we have proved the second part of (3).

If $\varphi \in \IBr {G \mid \theta}$, then $\varphi$ has a lift $\tau \in \irr {G \mid \theta}$.  Let $\varphi^* = (\tau^*)^o \in \IBr {\Gamma \mid \gamma}$, and observe that the map $\varphi \rightarrow \varphi^*$ is clearly independent of the choice of $\tau$, and thus well-defined.  Observe that if $\varphi, \sigma \in \IBr {G \mid \theta}$ so that $\varphi^* = \sigma^*$, then $\varphi^*_{H^*} = \sigma^*_{H^*}$.  By the character triple isomorphism, we have $\varphi_H = \sigma_H$, and it follows that $\varphi = \sigma$.  Hence, this map is injective.  Suppose $\kappa \in \IBr {\Gamma \mid \gamma}$.  Let $\zeta \in \irr {\Gamma \mid \gamma}$ be such that $\zeta^o = \kappa$.  Let $\eta \in \irr {G \mid \theta}$ be such that $\eta^* = \zeta$.  Since $\zeta^o$ is irreducible, we have seen that $\eta^o$ must be irreducible.  We see that $(\eta^o)^* = \kappa$, and so the map is surjective, thus proving statement (2).

It remains only to prove the first part of (3).  Suppose that $\chi^o$ is irreducible, and observe that $\eta \in \irr {G \mid \chi^o}$ if and only if $\eta_H = \chi_H + \Theta$, where $\Theta$ is a character of $H$ (or zero) that lies over $\theta$.  This occurs if and only if $\eta^*_{H^*} = \psi_{H^*} + \Theta^*$, and this occurs if and only if $\psi^o$ is a constituent of $(\eta^*)^o$.  This implies that the character triple isomorphism gives a bijection between $\irr {G : \chi^o}$ and $\irr {\Gamma : \psi^o}$.
\end{proof}

We will also make use of a result of Dade regarding the Glauberman correspondence found in \cite{dade}.  If $S$ is a solvable group acting coprimely on a group $K$, then we write $\irri SK$ for the characters in $\irr K$ that are left invariant by the action of $S$.  Recall from Chapter 13 of \cite{text} that Glauberman's correspondence yields a bijection from $\irri SK$ to $\irr {C_K (S)}$.  If $\chi \in \irri SK$, then we use $\chi^*$ to represent the Glauberman correspondent of $\chi$.

Now, suppose that $N$ is a normal subgroup of a group $G$, and that $M$ is a normal subgroup of $G$ that contains $N$ so that $N$ is a Hall subgroup of $M$ and $M/N$ is solvable.  By the Schur-Zassenhaus theorem, we know that $N$ has a complement $Q$ in $M$, and by the Frattini argument, $G = N N_G (Q)$.  It is not difficult to see that $N \cap N_G (Q) = C_N (Q)$.  In \cite{dade}, Dade outlines a proof of the fact that if $\theta \in \irr N$ is $G$-invariant and $\theta^* \in \irr {C_N (Q)}$ is the Glauberman correspondent of $\theta$, then $\theta$ and $\theta^*$ determine the same element in the Schur multiplier of of $G/N \cong N_G (Q)/C_N (Q)$.  (See Chapter 11 of \cite{text} for a discussion of Schur multipliers and characters determining elements in the Schur multiplier.)

Looking at the proof of Theorem 11.28 of \cite{text}, we see that if two characters determine the same element of the Schur multiplier, then the resulting character triples are isomorphic.  Hence, Dade's result can be read as saying that $(G,N,\theta)$ and $(N_G (Q), C_N (Q), \theta^*)$ are character triple isomorphic.  We note that Dade did not publish a complete proof of his result (it is proved in a preprint), but recently, Turull has published a paper with a stronger result (see \cite{turull}).

\begin{theorem}\label{normalizer}
Let $G$ be a $p$-solvable group, let $\varphi \in \IBr G$ lie in a block $B$ with abelian defect group $D$, and let $V$ be the Fong subgroup of $B$ and  $\widehat{\varphi} \in \IBr V$ be the Brauer character that induces to $\varphi$.  Then there exists a Brauer character $\widetilde\varphi \in \IBr {\norm_V (D)}$ that satisfies:
\begin{enumerate}
\item $(\widetilde\varphi)_{\cent_{{\bf O}_{p'} (V)} (D)}$ is homogeneous.
\item The unique irreducible constituent of $(\widetilde\varphi)_{\cent_{{\bf O}_{p'} (V)} (D)}$ corresponds to the unique irreducible constituent of $\widehat\varphi_{{\bf O}_{p'} (V)}$ via the Glauberman correspondence.
\item There is a bijection between $\irr {G : \varphi}$ and $\irr {\norm_V (D) : \widetilde\varphi}$ so that if $\widetilde\chi \in \irr {\norm_V (D) : \widetilde\varphi}$ corresponds to $\chi \in \irr {G : \varphi}$, then $\widetilde\chi(1)/\widetilde\varphi (1) = \chi (1)/\varphi (1)$ and thus the bijection restricts to a bijection between $L_\varphi$ and $L_{\widetilde\varphi}$.
\end{enumerate}
\end{theorem}

\begin{proof}
Let $N = {\bf O}_{p'} (V)$.  Let $\alpha$ be the unique irreducible constituent of $\widehat\varphi_N$.  We see that $D$ is a Sylow $p$-subgroup of $V$.  Let $\beta \in \irr {\cent_N (D)}$ be the Glauberman correspondent for $\alpha$.  Let $M/N = {\bf O}_p (V/N)$.  By the Hall-Higman theorem, $\cent_{V/N} (M/N) \le M/N$.  Since $D$ is abelian, this implies that $D \le M$, and so, $M = ND$.  We now use Dade's result mentioned before this lemma to provide a character triple isomorphism between $(V,N,\alpha)$ and $(\norm_V (D), \cent_N (D), \beta)$.  Using Lemma \ref{triples}, we can take $\widetilde{\varphi} = (\widehat{\varphi})^* \in \IBr {\norm_V (D)}$ to correspond to $\widehat\varphi$ under the character triple isomorphism, and by that lemma, we get bijections from $\irr {\norm_V (D) : \widetilde{\varphi}}$ to $\irr {V : \widehat\varphi}$ and from $L_{\widetilde{\varphi}}$ to $L_{\widehat\varphi}$.  Using Theorem \ref{blockcomparison}, we have bijections between $\irr {V : \widehat\varphi}$ and $\irr {G : \varphi}$ and between $L_{\widehat\varphi}$ and $L_\varphi$.  Combining these, we get the desired bijections.  Notice that both induction and the character triple isomorphism preserve the ratio of degrees, so the given ratios hold.
\end{proof}

\section{The proofs of Theorems \ref{necsuf} and \ref{3conditions}}

In this section we prove Theorems \ref{necsuf} and \ref{3conditions}.  One direction of Theorem \ref{necsuf} will follow easily from the theory we have established (and in fact will hold for $p$-solvable groups, or even $\pi$-separable groups in the $\pi$-theory).  The other direction will require a very different argument, and will require $G$ to be solvable.  However, we do not yet know of a counterexample to suggest that the solvable hypothesis cannot be replaced by $p$-solvable.

Our first lemma is essentially the ``invariant'' case of one direction of Theorem \ref{necsuf}.

\begin{lemma} \label{alpha invariant}
Let $p$ be a prime, let $G$ be a $p$-solvable group, let $N = O_{p'} (G)$, and let $\varphi \in \IBr G$ be contained in the block $B$ with defect group $D$.  Suppose $\alpha \in \irr N$ is an irreducible constituent of $\varphi_N$ and is $G$-invariant.  If $D \le Z(G)$, then $|L_\varphi| = |D|$.  Moreover, $\varphi$ is the unique Brauer character of $B$.
\end{lemma}

\begin{proof}
Since $\alpha$ is $G$-invariant, $D$ is a Sylow $p$-subgroup of $G$ by a theorem of Fong (Theorem 10.20 of \cite{blocks}).  Let $H$ be a Hall $p$-complement for $G$.  Since $D$ is a Sylow $p$-subgroup and central, we have that $G = H \times D$.  We see that $\varphi_H = \theta \in \irr H$, and the set of lifts of $\varphi$ is the set $\irr {G \mid \theta}$.  Since $|\irr {G \mid \theta}| = |D|$, this gives $|L_\varphi| = |D|$, as desired.  Finally, since $G$ is the direct product of a $p$-group and a $p'$-group, then $\varphi$ is the unique Brauer character in $B$.
\end{proof}

We now have as a corollary one direction of Theorem \ref{necsuf}.  In fact, we prove slightly more.

\begin{theorem}\label{onedirection}  Let $G$ be a $p$-solvable group and let $B$ be a block of $G$ with abelian defect group $D$, and suppose $\varphi \in \IBr {B}$.  Let $V$ be the Fong subgroup associated to $B$.  If $D$ is central in $\norm_V(D)$, then $|L_{\varphi}| = |D|$.  Moreover, in this case $\varphi$ is the unique Brauer character in $B$.
\end{theorem}

\begin{proof}
Let $\widetilde{\varphi} \in \IBr {\norm_G (V)}$ be as defined in Theorem \ref{normalizer}.  By Theorem \ref{alpha invariant}, $|L_{\widetilde{\varphi}}| = |D|$ and $\widetilde{\varphi}$ is the unique Brauer character in its block.  On the other hand, using  Theorem \ref{normalizer}, we have $|L_\varphi| = |L_{\widetilde{\varphi}}|$.  Combining these two equations, we have the desired equality.  The final statement follows from Lemmas \ref{fongred} and \ref{triples}.
\end{proof}

We point out that the conclusion that $\varphi$ is the unique Brauer character in $B$ also appears in \cite{Alperin}, though we include our proof here since it follows easily from the results we have developed.

We now begin working toward the proof of the other direction of Theorem \ref{necsuf}.  Fortunately, along the way we will prove most of Theorem \ref{3conditions}.  We note that the first conclusion of the next lemma follows from the $k(B)$-conjecture, but since our proof is simple, we decided to include it.

\begin{lemma} \label{all lifts}
Let $G$ be a $p$-solvable group, let $D$ be a normal, abelian Sylow $p$-subgroup, and let $\varphi \in \IBr G$.  Then the following are true:
\begin{enumerate}
\item $|\irr {G : \varphi}| \le |D|$.
\item If $|\irr {G : \varphi}| = |D|$, then $\irr {G : \varphi} = L_\varphi$.
\end{enumerate}
\end{lemma}

\begin{proof}
Let $H$ be a Hall $p$-complement of $G$.  Hence, we have $G = HD$.  Let $\chi \in B_{p'} (G)$ be a lift of $\varphi$ (see \cite{bpi} for a discussion of Isaacs' $B_{p'}$ characters, which are canonical lifts of the Brauer characters).  Then we know that $D$ is in the kernel of $\chi$, so $\chi_H = \varphi_H$ is irreducible.  If $\gamma \in \irr {G : \varphi}$, then $\varphi_H$ is a constituent of $\gamma_H$.  Let $\psi_1 = \chi, \psi_2, \dots \psi_m$ be the characters in $\irr {G : \varphi}$.  Then $(\varphi_H)^G = \sum_{i = 1}^m \psi_i + \Theta$ where $\Theta$ is either identically $0$ or a character of $G$.  Observe that
$$
|D| \varphi (1) = (\varphi_H)^G (1) = \sum_{i = 1}^m \psi_i (1) + \Theta (1) \ge m \varphi (1).
$$
This implies that $m \le |D|$, and (1) is proved.  Suppose $m = |D|$, then we must have equality throughout, and this implies that $\Theta (1) = 0$ and $\psi_i (1) = \varphi (1)$ for all $i$.  Therefore, $\Theta$ is identically $0$, and each of the $\psi_i$ is a lift of $\varphi$.  In particular, $\irr {G : \varphi}$ contains only lifts of $\varphi$.
\end{proof}

For a character $\delta$ of a normal subgroup $N$ of $G$, we will let $G_{\delta}$ denote the inertia subgroup of $\delta$.

\begin{lemma} \label{all orbits}
Let $G$ be a $p$-solvable group, and let $\varphi \in \IBr{G}$ be in a block with a normal, abelian defect group $D$.  Let $\delta \in \irr {D}$.  If $|L_\varphi| = |D|$, then $|L_\varphi \cap \irr {G \mid \delta}| = |G:G_\delta|$.
\end{lemma}

\begin{proof}
To prove this, we use Theorem 2.6 of \cite{CoLeNa}, and we adopt the notation from that result.  Hence, let ${\mathcal A}$ be a complete set of representatives of the $G$-orbits of $\irr D$ with $\delta \in {\mathcal A}$.  For $\lambda \in {\mathcal A}$, the set $\cent_\lambda$ defined there is contained in $L_\varphi \cap \irr {G \mid \lambda}$.  Since $|L_\varphi|$ equals both $\sum_{\lambda \in {\mathcal A}} |\cent_\lambda|$ and $\sum_{\lambda \in {\mathcal A}} |L_\varphi \cap \irr {G \mid \lambda}|$, it follows that $\cent_{\delta} = L_\varphi \cap \irr {G \mid \delta}$.  By Lemma 2.3 of \cite{CoLeNa}, we have $|L_\varphi \cap \irr {G \mid \delta}| \le |G:G_\delta|$.  Applying Theorem 2.6 of \cite{CoLeNa} again, we have
$$
|L_\varphi| = \sum_{\lambda \in {\mathcal A}} |L_\varphi \cap \irr {G \mid \lambda}| \le \sum_{\lambda \in {\mathcal A}} |G:G_\lambda| = |D|.
$$
Our hypothesis implies that we must have equality throughout, and so, $|L_\varphi \cap \irr {D \mid \delta}| = |G:G_\delta|$, as desired.
\end{proof}

We now look at the case when $|\irr {G : \varphi}| = |D|$.

\begin{theorem} \label{all block}
Let $G$ be a $p$-solvable group, and let $\varphi \in \IBr G$ lie in a block $B$ with abelian defect group $D$.  If $|\irr {G : \varphi}| = |D|$, then $\irr {G : \varphi} = L_\varphi$.  Moreover, $\IBr B = \{ \varphi \}$ and $\irr{B} = L_{\varphi}$.
\end{theorem}

\begin{proof}
Let $(V,\gamma)$ be a Fong pair for $\varphi$ so that $D \le V$.  Let $\widetilde{\varphi} \in \IBr {\norm_V (D)}$ correspond to $\varphi$ as in Theorem \ref{normalizer}.  By that theorem, we have $|L_{\widetilde{\varphi}}| = |L_\varphi|$ and $|\irr {G : \varphi}| = |\irr {\norm_V (D) : \widetilde{\varphi}}| = |D|$.  Notice that $\norm_V (D)$ now satisfies the hypotheses of Lemma \ref{all lifts}, and hence, $L_{\widetilde{\varphi}} = \irr {\norm_V (D) : \widetilde{\varphi}}$.  We now have $|\irr {G : \varphi}| = |\irr {\norm_V (D) : \widetilde{\varphi}}| = |L_{\widetilde{\varphi}}| = |L_\varphi|$.  Since $L_\varphi \subseteq \irr {G : \varphi}$, we obtain $L_\varphi = \irr {G \mid \varphi}$.  Notice that no irreducible character of $G$ has a restriction to $G^o$ with $\varphi$ and some other irreducible Brauer character as constituents.  Hence, $\IBr {B} = \{ \varphi \}$ and $\irr {B} = L_{\varphi}$.
\end{proof}

We should note here that in the next theorem, we need the stronger hypothesis of solvability rather than $p$-solvability that we have been assuming everywhere else.  We need the solvability hypothesis to appeal to Dolfi's theorem.

\begin{theorem} \label{all center}
Let $G$ be a solvable group, let $D$ be a normal, abelian Sylow $p$-subgroup, and let $\varphi \in \IBr G$.  Assume that $\varphi_{{\bf O}_{p'} (G)}$ is homogeneous.  If $|L_\varphi| = |D|$, then $D \le Z (G)$.
\end{theorem}

\begin{proof}
We work by induction on $|G|$.  Let $H$ be a Hall $p$-complement of $G$.  Hence, we have $G = HD$.  Let $\chi \in B_{p'} (G)$ be a lift of $\varphi$.  Then we know that $D$ is in the kernel of $\chi$, so $\chi_H = \varphi_H$ is irreducible.

Let $N = {\bf O}_{p'} (G)$.  If $ND = G$, then $G = N \times D$, and the result is trivial in this case.  Thus we assume that $ND < G$, and it not difficult to see that $ND = \cent_G (D)$, so $H/N$ acts faithfully on $D$.  Let $\alpha$ be the unique irreducible constituent of $\varphi_N$.  By Lemma \ref{all lifts}, we have $|L_{\varphi}| \leq |\irr {G : \varphi}| \leq |D|$, and thus we have equality throughout.  Thus by Theorem \ref{all block}, we have $\IBr {G \mid \alpha} = \{ \varphi \}$ and $\irr {G \mid \alpha} = \irr {G : \varphi}$ contains only lifts of $\varphi$.  We can now apply Lemma 4.2 of \cite{Alperin} to see that $\alpha$ is fully-ramified with respect to $H/N$, and thus $\alpha \times 1_D$ is fully-ramified with respect to $G/ND$.  Notice that $\chi$ must be the unique irreducible constituent of $(\alpha \times 1_D)^G$, and thus $(\chi (1)/\alpha (1))^2 = |G:ND| = |H:N|$.  Let $\chi (1) = e \alpha (1)$, so $\varphi (1) = \chi (1) = e \alpha (1)$ and $e^2 = |G:ND|$.

Suppose that there exists a normal subgroup $M$ of $G$ so that $1 < M < D$.  Let $\{ \delta_1, \dots, \delta_m \}$ be a complete set of representatives of the $H$ orbits of $\irr {D/M}$.  The number of lifts of $\varphi$ in $G/M$ will equal $$
\sum_{i=1}^m |L_\varphi \cap \irr {G \mid \delta_i}| = \sum_{i=1}^m |G:G_{\delta_i}| = |D:M|
$$
by Lemma \ref{all orbits} applied to the characters $\delta_i$.  Thus, $G/M$ satisfies the hypotheses of the theorem. By the inductive hypothesis, we have that $D/M$ is central in $G/M$, so $[D,G] \le M$.  This implies that $[D,H] \le M$.

Notice that $H$ centralizes $\cent_D (H)$ and $D$ centralizes $\cent_D (H)$ since $D$ is abelian.  Thus, $\cent_D (H)$ is normal in $HD = G$.  If $1 < \cent_D (H) < D$, then $[D,H] \le \cent_D (H)$ by the previous paragraph.  On the other hand, $D = [D,H] \times \cent_D (H)$ by Fitting's theorem.  This is a contradiction.  Thus, either $D = \cent_D (H)$ and $D \le Z (G)$ as desired, or $D = [D,H]$.  Thus, we assume that $D = [D,H]$, and by the previous paragraph, this implies that $D$ is minimal normal in $G$ since otherwise, $[D,H] < D$.

Using a theorem of Dolfi (Theorem 1.1 in \cite{dolfi}), we can find $\delta, \sigma \in \irr D$ so that $\cent_H (\delta) \cap \cent_H (\sigma) = N$.  Observe that $G_{\alpha \times \delta} = G_\delta = \cent_H (\delta) D$ and $G_{\alpha \times \sigma} = G_\sigma = \cent_H (\sigma) D$.  Since all the characters in $\irr {G \mid \alpha \times \delta}$ are lifts of $\varphi$, we have that $|G:G_\delta|$ divides $\varphi (1)/\alpha (1) = e$.  Since $G_{\delta} \cap G_{\alpha \times \sigma} = ND$, we have that $(\alpha \times \sigma)^{G_\delta}$ is irreducible.  Also, the characters in $\irr {G \mid (\alpha \times \sigma)^{G_\delta}}$ are lifts of $\varphi$, so $(\alpha \times \sigma)^{G_\delta} (1) = |G_\delta:ND| \alpha (1) \le \varphi (1) = e \alpha (1)$.  This implies that $|G_\delta:ND| \le e$.  We now have
$$
|G:ND| = |G:G_\delta| |G_\delta:ND| \le e \cdot e = e^2 = |G:ND|,
$$
and thus we must have equality throughout.  This implies that $|G:G_\delta| = |G_\delta:ND| = e$.  Since all the characters in $\irr {G \mid \alpha \times \delta}$ have degree $e \alpha (1)$, we see that all of the characters in $\irr {G_\delta \mid \alpha \times \delta}$ have degree $\alpha (1)$ and so, $\irr {G_\delta \mid \alpha \times \delta}$ contains only extensions of $\alpha \times \delta$.  By Gallagher's theorem, this implies that $G_\delta/ND$ is abelian.  In a similar fashion, $|G_\sigma:ND| = e$ and $G_\sigma/ND$ is abelian.  Since
$$
|G_\delta G_\sigma:ND| = \frac {|G_\delta:ND| |G_\sigma:ND|}{ |(G_\delta \cap G_\sigma):ND|} = \frac {e \cdot e}1 = e^2 = |G:ND|,
$$
we have $G = G_\delta G_\sigma$.  Therefore, $G/ND$ is the product of two abelian subgroups.

Finally, It\^o has proved that if a nontrivial group is the product of two abelian subgroups, then one of those subgroups must contain a nontrivial normal subgroup of the group (see Satz 2 of \cite{ito}).  Thus, we can find $K$ normal in $G$ so that $NM < K \le G_\delta$.  This implies that $\cent_D (K) > 1$.  Since $D$ and $K$ are normal in $G$, we know that $\cent_D (K)$ is normal in $G$.  Because $D$ is minimal normal in $G$, we deduce that $D = \cent_D (K)$, but this implies that $K \le \cent_G (D) = ND$, a contradiction.
\end{proof}

We can now finally prove the other direction of Theorem \ref{necsuf}.

\begin{corollary}
Let $G$ be a solvable group, and let $\varphi \in \IBr G$ lie in a block $B$ with abelian defect group $D$.  Let $V$ be a Fong subgroup for $B$ containing $D$.  If $|L_\varphi| = |D|$, then $D \le Z (\norm_V (D))$.
\end{corollary}

\begin{proof}
Let $\chi$ be a lift of $\varphi$ and let $(V,\gamma)$ be a Fong pair for $(G,\chi)$ that contains $D$.    Let $\widetilde{\varphi} \in \IBr {\norm_V (D)}$ correspond to $\varphi$ as in Theorem \ref{normalizer}.  By that theorem, we have $|L_{\widetilde{\varphi}}| = |L_\varphi| = |D|$.  Notice that $\norm_V (D)$ now satisfies the hypotheses of Theorem \ref{all center}, and by that theorem, we have $D \le Z (\norm_V (D))$.
\end{proof}

We are now almost ready to prove Theorem \ref{3conditions}.  We first need two easy results.  Interestingly, in the next lemma we conclude that the defect group is abelian if the defect group is normal and every character in $\irr {G :\varphi}$ is a lift of $\varphi$.

\begin{lemma} \label{all lifts con}
Let $G$ be a $p$-solvable group, let $D$ be a normal Sylow $p$-subgroup, and let $\varphi \in \IBr G$.  If $\irr {G : \varphi} = L_\varphi$, then $|L_\varphi| = |D|$ and $D$ is abelian.
\end{lemma}

\begin{proof}
Let $H$ be a Hall $p$-complement of $G$, so that we have $G = HD$.  Let $\chi \in B_{p'} (G)$ be a lift of $\varphi$.  Then we know that $D$ is in the kernel of $\chi$, so $\chi_H = \varphi_H$ is irreducible.  If $\gamma \in \irr {G : \varphi}$, then $\gamma$ is a lift of $\varphi$, so $\gamma_H = \varphi_H$, and by Frobenius reciprocity, $\gamma$ has multiplicity $1$ as a constituent of $(\varphi_H)^G$.  Let $\psi_1, \psi_2, \dots \psi_m$ be the lifts of $\varphi$ where $m = |L_\varphi|$.  Then $(\varphi_H)^G = \sum_{i = 1}^{m} \psi_i$, and therefore $(\varphi_H)^G (1) = |G:H| \varphi (1) = |D| \varphi (1)$ and $(\varphi_H)^G (1) = \sum_{i=1}^m \psi_i (1) = m \varphi (1)$.  This implies that $m = |D|$.

Let $N = {\bf O}_{p'} (G)$, and let $\alpha \in \irr N$ be a constituent of $\varphi_N$.  Let $\varphi_\alpha \in \IBr {G_\alpha \mid \alpha}$ be the Clifford correspondent for $\varphi$.  By Clifford induction, a character $\chi \in \irr{G}$ is a lift of $\varphi$ if and only if $\chi$ is induced from a lift of $\varphi_{\alpha}$.  Moreover, a character $\xi \in \irr{G}$ is in $\irr {G : \varphi}$ if and only if it lies over a character in $\irr {G_{\alpha} : \varphi_{\alpha}}$.  Thus we have $L_{\varphi_\alpha} = \irr {G_\alpha : \varphi_\alpha}$.  Hence, it suffices to assume that $\alpha$ is $G$-invariant.  This implies that $\irr {G : \varphi} = \irr {G \mid \alpha}$.  If $\delta \in \irr D$, this implies that $\irr {G \mid \alpha \times \delta}$ consists of lifts of $\varphi$.  If $\psi \in \irr {G \mid \alpha \times \delta}$, then $\psi (1) = \varphi (1)$ is not divisible by $p$.  Since $\delta (1)$ divides $\psi (1)$, we conclude that $\delta (1) = 1$, so all irreducible characters of $D$ are linear.  Therefore, $D$ is abelian.
\end{proof}

This proves one assertion of Theorem \ref{3conditions}.

\begin{corollary} \label{finalcorollary}
Let $G$ be a $p$-solvable group, and let $\varphi \in \IBr G$ lie in a block $B$ with abelian defect group $D$.  If $\irr {G \mid \varphi} = L_\varphi$, then $|L_\varphi| = |D|$.
\end{corollary}

\begin{proof}
Let $V$ be the Fong subgroup for the block $B$, and let $\widetilde{\varphi} \in \IBr {\norm_V(D)}$ be the character corresponding to $\varphi$ as in Theorem \ref{normalizer}.  By that theorem, we have $|L_{\widetilde{\varphi}}| = |L_\varphi|$ and $|\irr {\norm_V (D) : \widetilde{\varphi}}| = |\irr {G : \varphi}|$.  This implies that $|L_{\widetilde{\varphi}}| = |\irr {\norm_V (D) : \widetilde{\varphi}}|$, and so, $L_{\widetilde{\varphi}} = \irr {\norm_V (D) : \widetilde{\varphi}}$.  We now apply Lemma \ref{all lifts con} to see that $|L_{\widetilde{\varphi}}| = |D|$, and the equality $|L_\varphi| = |L_{\widetilde{\varphi}}| = |D|$ yields the result.
\end{proof}

We now prove Theorem \ref{3conditions}.

\begin{proof}[Proof of Theorem \ref{3conditions}]  Let $V$ be the Fong subgroup for the block containing $\varphi$, and let $\widetilde{\varphi}$ be the character of $\norm_V(D)$ corresponding to $\varphi$.  By Theorem \ref{normalizer}, it is enough to prove the theorem for $\norm_V(D)$ and $\widetilde{\varphi}$, so by replacing $G$ with $\norm_V(D)$, we may assume that $D$ is normal in $G$.

Assume (1), that $|L_{\varphi}| = |D|$.  By Lemma \ref{all lifts}, we have that $|L_{\varphi}| \leq |\irr {G : \varphi}| \leq |D|$, and thus we have equality throughout, proving (2).

Now assume (2), that every character in $\irr {G : \varphi}$ is a lift of $\varphi$.  By Corollary \ref{finalcorollary}, we have that $|D| = |L_{\varphi}| = |\irr {G : \varphi}|$, proving (3).

Finally, if we assume (3), then (1) follows immediately from \ref{all lifts}, since $|\irr {G : \varphi}| \leq |L_{\varphi}| \leq |D|$.
\end{proof}

\section{Example}

This section contains a counterexample showing that Theorem \ref{necsuf} is not true if we replace $\norm_V(D)$ with $\norm_G(D)$.

Let $G$ be the semi-direct product of a group of order 3 acting on a group
of order 91 as a fixed-point free automorphism, so $G$ is a Frobenius
group of order 273.  Take $p = 7$.  Notice that $G$ has 3 linear characters
and the remaining (30) irreducible characters of $G$ have degree 3.  Let $D$
be the subgroup of order 7, and notice that $D$ is normal in $G$ but not
central.  The Brauer characters of $G$ correspond to the characters of
$G/D$.  Let $\chi \in \irr {G/D}$ have degree 3.  Let $\varphi = \chi^o$.  Let $C$ be the cyclic subgroup of order 91, and let $\alpha$ be an irreducible
constituent of $\chi_C$.  Observe that $\alpha^G = \chi$ and $D$ is in the
kernel of $\alpha$.  Let $B$ be the subgroup of order 13, and observe that $C = B \times D$.  Let $\alpha_0 = \alpha_B$, and observe that $\alpha = \alpha_0 \times 1_D$.  Note that $C$ is the stabilizer for $\alpha_0$ in
$G$.  Observe that for each $\delta \in \irr D$, we have $(\alpha_0 \times
\delta)^G$ is a lift of $\varphi$.  By Gallagher's theorem and Clifford's
theorem, these are distinct for distinct $\delta$.  This gives $|D| = 7$
different lifts of $\varphi$.  But $D$ is not central in its normalizer.

\section{Nilpotent blocks}

Recall that all of the results and their proofs so far in this paper also hold in the case of Isaacs' $\pi$-partial characters.  In this section we will put our results into the context of nilpotent blocks.  Much is known about nilpotent blocks in the \lq \lq classical" case (i.e. where the set of primes $\pi$ is the complement of the prime $p$).  In this section we briefly sketch a proof of how our arguments yield a character-theoretic proof of a special case of the Broue-Puig theorem about nilpotent blocks, and thus at least a special case of the Broue-Puig theorem holds for $\pi$-blocks of solvable groups.

Throughout this section, $G$ is a solvable group and $\pi$ is a set of primes.  (Many, though not all, of the results in this section hold for $\pi$-separable groups rather than just solvable groups.)

\begin{definition} \label{nilpotentdef}  Let $Q$ be a $\pi'$-subgroup of $G$, and let $B$ be a $\pi$-block of $G$.  If $b$ is a block of $Q \cent_G(Q)$, we say the pair $(Q, b)$ is a subpair of $G$.  If $b^G = B$, we say the pair $(Q, b)$ is a $B$-subgroup of $G$.  Let $\norm_G(Q, b)$ denote the stabilizer in $\norm_G(Q)$ of the block $b$.

We say the block $B$ of $G$ is nilpotent if $\norm_G(Q, b) / \cent_G(Q)$ is a $\pi'$-group for every $B$-subgroup of $G$.

\end{definition}

The Broue-Puig theorem \cite{broue-puig} shows that if $B$ is a nilpotent $p$-block of a finite group $G$ with defect group $D$, then $l(B) = 1$ and $k(B) = |D|$.  We now give a brief sketch of a proof of a similar result for nilpotent $\pi$-blocks of solvable groups with abelian defect group.

\begin{theorem}\label{nilpotentpi}  Let $G$ be a solvable group, and suppose $B$ is a nilpotent $\pi$-block of $G$ with an abelian defect group $D$.  Then there is a unique $\pi$-partial character $\varphi$ in $B$, and there are exactly $|D|$ ordinary irreducible characters in $D$, all of which lift $\varphi$.
\end{theorem}

\begin{proof}[Sketch of Proof]  We will show that $D$ is central in $\norm_V(D)$, where $V$ is the Fong subgroup for $V$.  Thus let $V$ be the Fong subgroup for $B$ with corresponding block $B_1$, and note that if $M = \opi(V)$, then $B_1$ covers a unique character $\alpha$ of $M$, and (after conjugating, if necessary), $D$ is a Hall $\pi'$-subgroup of $V$.

One can show that if $(\cent_G(D), b)$ is a root of $B$ (see \cite{cyclic} for a discussion of roots of $\pi$-blocks with abelian defect group), then $(D, b)$ is a $B$-subgroup, and thus $\norm_G(D, b)/\cent_G(D)$ is a $\pi'$-group.  Thus if $(\cent_V(D), b_1)$ is the root of $B_1$, then $\norm_V(D, b_1)/\cent_V(D)$ is a $\pi'$-group.  (See \cite{cyclic} for more details of this argument, here we are essentially repeatedly applying Lemma 5.1 of \cite{cyclic} to go from the root $b$ of $B$ to the root $b_1$ of $B_1$.)

Now $D$ acts coprimely on $M$, and let $\beta \in \irr{\cent_M(D)}$ be the Glauberman correspondent of $\alpha$.  One can show that $\cent_V(D) = D \cent_M(D)$ and thus $\beta$ extends to $\widehat{\beta} \in \irr{\cent_V(D)}$, and since $\alpha$ is invariant in $V$, then $\beta$ and $\widehat{\beta}$ are invariant in $\norm_V(D)$.  Now $\widehat{\beta}$ is the unique character in $b_1$, and we have $\norm_V(D, b_1) = I_{\norm_V(D)}(\widehat{\beta}) = \norm_V(D)$.  Since $D$ is a Hall $\pi'$-subgroup of $V$, then $\norm_V(D)/\cent_V(D)$ is a a $\pi$-group.  But we have already shown that $\norm_V(D)/\cent_V(D)$ is a $\pi'$-group.  Thus $\cent_V(D) = \norm_V(D)$, and thus $D$ is central in $\norm_V(D)$.  By Theorem \ref{necsuf}, we have that $|L_{\varphi}| = |D|$, and by Theorem \ref{3conditions}, we are done.
\end{proof}

Notice that in the \lq \lq classical" case (where $\pi$ is the complement of the prime $p$), Theorem 4.1 of \cite{ehzd} shows that the conditions in Theorem \ref{3conditions} imply that $B$ is nilpotent.  In the $\pi$-case, we see that the \lq \lq hard" direction of Theorem \ref{necsuf} (which required the use of a large orbit theorem) shows that the conditions of Theorem \ref{3conditions} imply that $D$ is central in $\norm_V(D)$, though we do not know yet if this is equivalent to $B$ being nilpotent.

Finally, we mention that it seems reasonable that stronger version of the Broue-Puig theorem should hold in the $\pi$ case, without the requirement that the defect group $D$ is abelian.  That is, if $B$ is a nilpotent $\pi$-block of a solvable group $G$, it should be the case that there is a unique $\pi$-partial character $\varphi$ in $B$ and there should be exactly $|D|$ ordinary irreducible characters in $B$, all of which lift $\varphi$.  However, we do not yet have a proof of this.


\begin{thebibliography}{99}
\bibitem{broue-puig} M.~Broue and L.~Puig, A Frobenius theorem for blocks, {\it Invent. Math} {\bf 56} (1980), 117-128.
\bibitem{bounds} J.~P.~Cossey, Bounds on the number of lifts of a Brauer character in a $p$-solvable group, {\it J. Algebra} {\bf 312} (2007), 699-708.
\bibitem{cyclic} J.~P.~Cossey and M.~L.~Lewis, Lifts of partial characters with cyclic defect group, {\it J. Aust. Math. Soc.} {\bf 89} (2010), 145-163.
\bibitem{CoLeNa} J.~P.~Cossey, M.~.L.~Lewis, and G.~Navarro, The number of lifts of a Brauer character with a normal vertex, {\it J. Algebra} {\bf 328} (2011), 484–487.
\bibitem{dade} E.~C.Dade, A correspondence of characters, {\it Proc. Symp. Pure Math.} {\bf 37} (1980), 401-403.
\bibitem{dolfi} S.~Dolfi, Large orbits in coprime actions of solvable groups, {\it Trans. Amer. Math. Soc.} {\bf 360} (2008), 135-152.
\bibitem{bpi} I.~M.~Isaacs, Characters of $\pi$-separable groups, {\it J. Algebra} {\bf 86} (1984), 98-128.
\bibitem{text}  I.~M.~Isaacs, ``Character Theory of Finite Groups,'' Academic Press, San Diego, 1976.
\bibitem{ito} N.~It\^o, Uber das Produkt von zwei abelschen Gruppen, {\it Math. Z.} {\bf 63} (1955), 400-401.
\bibitem{ehzd} G.~Malle and G.~Navarro, Blocks with equal height zero degrees, {\it Trans. Amer. Math. Soc} {\bf 363} (2011), 6647-6669.
\bibitem{Alperin} G.~Navarro, The Alperin argument revisited, {\it Osaka J. Math.} {\bf 33} (1996), 971-982.
\bibitem{blocks} G.~Navarro, {\it Characters and blocks of finite groups}, London Math Society Lecture Notes 250, Cambridge University Press, Cambridge, 1998.
\bibitem{slattery1} M.~C.~Slattery, Pi-blocks of pi-separable
    groups, I, \textit{J. Algebra} \textbf{102} (1986), 60-77.

\bibitem{slattery2} M.~C.~Slattery, Pi-blocks of pi-separable
    groups, II, \textit{J. Algebra} \textbf{124}

    (1989), 236-269.
\bibitem{turull} A.~Turull, Above the Glauberman correspondence, {\it Adv. Math.} {\bf 217} (2008), 2170–2205.


























\end{thebibliography}
\end{document}